\documentclass[11pt]{amsart}

\usepackage{amsmath}
\usepackage{amssymb}
\usepackage{bm}
\usepackage{graphicx}
\usepackage{psfrag}
\usepackage{color}
\usepackage{hyperref}
\hypersetup{colorlinks=true, linkcolor=blue, citecolor=magenta, urlcolor=wine}
\usepackage{url}
\usepackage{algpseudocode}
\usepackage{fancyhdr}
\usepackage{mathtools}
\usepackage{tikz-cd}
\usepackage{xy}
\input xy
\xyoption{all}
\usepackage{stmaryrd}
\usepackage{calrsfs}

\newtheorem*{maintheorem*}{Main Theorem}
\newtheorem{theorem}{Theorem}[section]
\newtheorem{prop}[theorem]{Proposition}

\theoremstyle{definition}
\newtheorem{defn}[theorem]{Definition}

\newtheorem{example}[theorem]{Example}

\newtheorem{question}[theorem]{Question}
\numberwithin{equation}{section}

\newcommand{\cc}{\mathbb{C}}
\newcommand{\ff}{\mathbb{F}}
\newcommand{\nn}{\mathbb{N}}
\newcommand{\pp}{\mathbb{P}}

\newcommand{\qq}{\mathbb{Q}}
\newcommand{\rr}{\mathbb{R}}
\newcommand{\zz}{\mathbb{Z}}

\providecommand\ldb{\llbracket}
\providecommand\rdb{\rrbracket}

\newcommand{\mcd}{\mathrm{mcd}}
\newcommand{\uu}{\mathcal{U}}

\keywords{commutative ring, polynomial ring, IDF-domain, GL-domain, GCD-domain}

\subjclass[2020]{Primary: 13A05, 13F15; Secondary: 13P05}

\begin{document}
	
\mbox{}
\title{On GL-domains and the ascent of the IDF property}

\author{Victor Gonzalez}
\address{Department of Mathematics\\Miami Dade College\\Miami, FL 33135, USA}
\email{vmichelg@mdc.edu}

\author{Ishan Panpaliya}
\address{Department of Mathematics\\Seattle University\\Seattle, WA 02139, US}
\email{ipanpaliya@seattleu.edu}

\date{\today}

\begin{abstract}

    Following the terminology introduced by Arnold and Sheldon back in 1975, we say that an integral domain $D$ is a GL-domain if the product of any two primitive polynomials over $D$ is again a primitive polynomial. In this paper, we study the class of GL-domains. First, we propose a characterization of GL-domains in terms of certain elements we call prime-like. Then we identify a new class of GL-domains. An integral domain $D$ is also said to have the IDF property provided that each nonzero element of $D$ is divisible by only finitely many non-associate irreducible divisors. It was proved by Malcolmson and Okoh in 2009 that the IDF property ascends to polynomial extensions when restricted to the class of GCD-domains. This result was recently strengthened by Gotti and Zafrullah to the class of PSP-domains. We conclude this paper by proving that the IDF property does not ascend to polynomial extensions in the class of GL-domains, answering an open question posed by Gotti and Zafrullah.
\end{abstract}

\bigskip
\maketitle

\bigskip
\section{Introduction}
\label{sec:intro}

Let $D$ be an integral domain with multiplicative group of units denoted by $D^\times$, and let $F$ be the field of fractions of $D$. An ideal $I$ of $D$ is said to be invertible if $II^{-1} = D$, where $I^{-1} := \{q \in F : qI \subseteq D\}$. If every nonempty finite subset of $D$ consisting of nonzero elements has a greatest common divisor (GCD for short) in $D$, then we say that $D$ is a GCD-domain. Given a nonzero polynomial $f \in D[X]$, the ideal of $D[X]$ generated by all the coefficients of $f$ is referred to as the content of $f$ and is denoted by $I_f$. Then we have the following terminology for the polynomial $f$ in terms of its content:
\begin{itemize}
	\item $f$ is a primitive polynomial over $D$ if $I_f$ is not contained in any proper principal ideal of $D$, and
	\smallskip
	
	\item $f$ is a super-primitive polynomial over $D$ if $I_f^{-1} = D$.
\end{itemize}
Observe that $f$ is primitive over $D$ if the only common divisors of its coefficients are the units in $D$. Gauss's lemma states that, in a GCD-domain, the product of two primitive polynomials is again a primitive polynomial. With Gauss's lemma as motivation, the following definition was introduced by J. Arnold and P. Sheldon~\cite{AS75} back in 1975.

\begin{defn} [J. Arnold and P. Sheldon~\cite{AS75}, 1975]
	An integral domain $D$ is called a \emph{GL-domain} if the product of two primitive polynomials in $D$ remains a primitive polynomial. 
\end{defn}
When an integral domain is a GL-domain, we also say it has the GL-property. By virtue of Gauss's lemma, every GCD-domain is a GL-domain. It is natural to wonder whether the class consisting of all GL-domains \emph{strictly} contains the class of all GCD-domains and, that being the case, which other relevant integral domains are GL-domains. These questions were addressed back in the seventies, first by H. Tang~\cite{hT72} and then by J. Arnold and P. Sheldon~\cite{AS75}, who studied the class of GL-domains in connection to the class of PSP-domains.

\begin{defn} [J. Arnold and P. Sheldon~\cite{AS75}, 1975]
	An integral domain $D$ is called a \emph{PSP-domain} if every primitive polynomial over $D$ is super-primitive. 
\end{defn}

The first investigation of PSP-domains was carried out by H. Tang~\cite{hT72} in his study of super-primitive polynomials. Among other results, he proved that every GCD-domain is a PSP-domain and also that in any integral domain the product of a super-primitive polynomial and a super-primitive (resp., primitive) polynomial is super-primitive (resp., primitive), which implies that the class consisting of all PSP-domains is contained in the class of all GL-domains. Several generalizations of GCD-domains, including GL-domains and PSP-domains, have been investigated in the literature of classical commutative ring theory (see \cite{dA00} and references therein).
\smallskip

In this paper, we first provide some insight into the class of all GL-domains, and then we tackle the still open question of whether every polynomial ring of a GL-domain with the irreducible-divisor-finite property (IDF property) preserves the IDF property.
\smallskip

\begin{defn} [A. Grams and H.~Warner~\cite{GW75}, 1975]
    An integral domain $D$ is an \emph{IDF domain} (or has the \emph{IDF property}) if every nonzero element of $D$ is divisible by only finitely many irreducibles up to associates.
\end{defn}

The question of whether the IDF property ascends to polynomial extensions was first posed by D. Anderson, D. Anderson, and M. Zafrullah~\cite{AAZ90} back in 1990. It was answered negatively by P. Malcolmson and F. Okoh in~\cite{MO09}. It is also known that the IDF property does not ascend to monoid-domain extensions (see \cite[Example~5.5]{fG22}). Also in~\cite{MO09}, P. Malcolmson and F. Okoh proved that the IDF property ascends to polynomial extensions if one restricts attention to the class of GCD-domains. This has been recently improved in two different directions. First, S. Eftekhari and M. Khorsandi~\cite{EK18} proved that the IDF property ascends to polynomial extensions over the class of MCD-finite domains. Second, and more recently, F. Gotti and M. Zafrullah~\cite{GZ23} proved that the IDF property ascends to polynomial extensions over the class of PSP-domains. The most natural generalization of the class consisting of all PSP-domains is the class consisting of all GL-domains, so it is natural to wonder whether the IDF property ascends to polynomial extensions when restricted to the class of GL-domains.

\begin{question} \label{quest:ascent of the IDF prop over GL-domains}
    For a GL-domain $D$, is it true that if $D$ has the IDF property, then the polynomial extension $D[x]$ also has the IDF property? 
\end{question}

In Section~\ref{sec:background}, which is the background section, we discuss some preliminaries, providing the notation, terminology, and standard results that we shall use later.
\smallskip

In Section~\ref{sec:some results on GL-domains}, which is the first section of content of this paper, we begin our investigation of the class of GL-domains. First, we characterize GL-domains in terms of certain elements whose behavior is somehow similar to that of a prime: an integral domain is a GL-domain if and only if every nonzero nonunit element is prime-like (a nonzero nonunit $p \in D$ is prime-like provided that for any $r,s \in D$ such that $p \mid_D rs$ there exists a nonunit divisor $d \in D$ of $p$ such that either $d \mid_D r$ or $d \mid_D s$). Then we prove that the notions of a GL-domain and a PSP-domain are equivalent when we restrict to the class of MCD-domains (an integral domain $D$ is MCD provided that any finite subset of $D$ consisting of nonzero elements has a maximal common divisor). We conclude Section~\ref{sec:some results on GL-domains} identifying a new class of GL-domains.
\smallskip

In Section~\ref{sec:ascent of the IDF-property over GL-domains}, which is the last section of this paper, we answer Question~\ref{quest:ascent of the IDF prop over GL-domains} negatively, providing an explicit construction of a GL-domain $D$ having the IDF property whose polynomial extension does not have the IDF property.

\bigskip
\section{Background}
\label{sec:background}

We let $\zz$, $\qq$, $\rr$, and $\cc$ be the set of integers, rational numbers, real numbers, and complex numbers, respectively. We let $\nn$ and $\nn_0$ be the set of positive and nonnegative integers, respectively. Also, we let $\pp$ denote the set of primes. For $p \in \pp$ and $n \in \nn$, we let $\ff_{p^n}$ be the finite field of cardinality $p^n$.
\smallskip

A commutative monoid is a commutative semigroup with an identity element. Let $M$ be a commutative monoid. The \emph{group of units} of $M$ is the abelian group $\uu(M)$ consisting of all invertible elements of $M$. 
\smallskip

Let $I$ be a nonempty set of indices, and let $\{M_i : i \in I \}$ be a collection of monoids. Then $P := \prod_{i \in I} M_i$ is also a monoid under the component-wise operation. We write the elements of the product monoid $P$ as $(b_i)_{i \in I}$. When $M_i = M$ for all $i \in I$, the monoid $P$ is denoted by $M^I$. For sets of indices $I$ and $J$ with $|I| = |J|$, one can check that $M^I$ and $M^J$ are isomorphic. Therefore when $I$ is countable, we often write $M^\nn$ instead of $M^I$.
\smallskip

For $b,c \in M$, we say that $c$ \emph{divides} $b$ if there exists $d \in M$ such that $b = cd$, in which case we write $c \mid_M b$. We say that $M$ is a \emph{valuation monoid} if for any $b,c \in M$ either $b \mid_M c$ or $c \mid_M b$. If for $b,c \in M$, both relations $b \mid_M c$ and $c \mid_M b$ hold, then we say that $b$ and $c$ are \emph{associates} and write $b \sim c$. Let $S$ be a nonempty subset of $M$. We say that $d \in M$ is a \emph{greatest common divisor} (GCD) of $S$ if $d$ is a common divisor of $S$ and for any other common divisor $d' \in M$ of $S$ the relation $d' \mid_M d$ holds. One can easily check that any two GCDs of $S$ are associates. We let $\gcd_M(S)$ or, simply $\gcd(S)$, denote the set consisting of all the GCDs of $S$ in $M$. If every nonempty finite subset of $M$ has a GCD, then $M$ is a \emph{GCD-monoid}. Observe that an integral domain $D$ is a GCD-domain if and only if its multiplicative monoid
\[
    D^* := \{d \in D : d \neq 0\}
\]
is a GCD-monoid. One can readily check that every valuation monoid is a GCD-monoid and also that the product of GCD-monoids is a GCD-monoid.
\smallskip

We say that $d$ is a \emph{maximal common divisor} (MCD) of $S$ if $d$ is a common divisor of $S$ and for any other common divisor $d' \in M$ of $S$ the divisibility relation $d \mid_M d'$ implies that $d$ and $d'$ are associates. Clearly, every GCD of $S$ is an MCD. We let $\mcd_M(S)$ denote the set consisting of all MCDs of $S$ in $M$. We say that $M$ is an \emph{MCD-monoid} provided that every nonempty finite subset of $M$ has an MCD. Clearly, every GCD-monoid is an MCD-monoid. Following~\cite{EK18}, we say that $M$ is \emph{MCD-finite} if every finite subset of $M$ has only finitely many non-associate maximal common divisors. 
\smallskip

For each $n \in \nn$, an element $b \in M$ is called $n$-\emph{divisible} if there exists $c \in M$ such that $nc = b$. The monoid $M$ is said to be $n$-\emph{divisible} for some $n \in \nn$ if every element of $M$ is $n$-divisible. Observe that if $M$ is $n$-divisible, then the product monoid $M^I$ is also $n$-divisible for every nonempty set of indices~$I$.
\smallskip

Let $D$ be an integral domain. Given a nonzero polynomial $f(x) \in D[x]$, the \emph{content} of $f(x)$ is the ideal generated by the coefficients of $f(x)$, and the polynomial $f(x)$ is called \emph{primitive} if its content is the whole integral domain $D$. As formally stated in the introduction, $D$ is called a GL-domain if the product of two primitive polynomials is again a primitive polynomial. As an immediate consequence of Gauss's lemma, one obtains that every GCD-domain is a GL-domain.

\bigskip
\section{Identifying Classes of GL-Domains}
\label{sec:some results on GL-domains}

In this section we provide a new characterization of GL-domains, characterize GL-domains inside of the class of MCD-domains, and identify a new class of GL-domains.

\medskip
\subsection{A Characterization of GL-domains}

The following characterization of a GL-domain is due to D. Anderson and R. Quintero~\cite{AQ97}.

\begin{prop} \cite[Theorem 3.1]{AQ97}
	For an integral domain $D$, the following conditions are equivalent.
	\begin{enumerate}
		\item[(a)] The integral domain $D$ is a GL-domain.
		\smallskip
		
		\item[(b)] For all $r,s,t \in D$, if $\gcd(r,s) = \gcd(r,t) = 1$ then $\gcd(r,st) = 1$.
	\end{enumerate} 
\end{prop}
\smallskip

In their paper~\cite{AQ97}, the authors provided further characterizations of GL-domains. Here we add one more characterization in terms of certain elements that somehow behave like primes, and so we call them prime-like.

\begin{defn}
	Let $D$ be an integral domain. A nonzero nonunit $p \in D$ is called \emph{prime-like} if the following property holds: for all $r,s \in D$ the fact that $p \mid_D rs$ ensures the existence of a divisor $d \in D \setminus D^\times$ of $p$ such that $d \mid_D r$ or $d \mid_D s$.
\end{defn}

We can now use the notion of prime-like elements to characterize GL-domains.

\begin{prop} \label{prop:prime-like characterization of GL-monoids}
	Let $D$ be an integral domain. Then the following conditions are equivalent.
	\begin{enumerate}
		\item[(a)] $D$ is a GL-domain.
		\smallskip
		
		\item[(b)] Each nonzero nonunit element of $D$ is prime-like.
	\end{enumerate}
\end{prop}

\begin{proof}
	(a) $\Rightarrow$ (b): First, assume that $D$ has the GL-property. Fix an element $p \in D \setminus D^\times$, and let us check that $p$ is prime-like. To do so, suppose that $p \mid_D rs$ for some $r,s \in D \setminus D^\times$. Since $p \mid_D rs$ and $p \notin D^\times$, it follows that $\gcd(p,rs) \neq 1$. This, along with the fact that $D$ satisfies the GL-property, ensures that either $\gcd(p,r) \neq 1$ or $\gcd(p,s) \neq 1$. Finally, observe that if $\gcd(p,r) \neq 1$ (resp., $\gcd(p,s) \neq 1$), then there exists a nonunit divisor of $p$ that also divides $r$ (resp., $s$). Hence $p$ is prime-like.
	\smallskip
	
	(b) $\Rightarrow$ (a): Assume now that every nonzero nonunit of $D$ is prime-like. Suppose, toward a contradiction, that $D$ does not have the GL-property. Thus, there exist $p,r,s \in D$ such that $\gcd(p,r) = \gcd(p,s) = 1$ but $\gcd(p,rs) \neq 1$. As $\gcd(p,rs) \neq 1$, we can take a nonunit $d \in D$ such that $d \mid_D p$ and $d \mid_D rs$. Because every nonzero nonunit of $D$ is prime-like, there is a nonunit $d' \in D$ dividing $d$ in $D$ such that either $d' \mid_D r$ or $d' \mid_D s$. After assuming, without loss of generality, that $d' \mid_D r$, we see that $d' \mid p$, whence $\gcd(p,r) \neq 1$, which is a contradiction.
\end{proof}

\medskip
\subsection{Connection with Maximal Common Divisors}

As the class of PSP-domains, it is well known that the class consisting of all pre-Schreier domains is contained in the class of GL-domains. In our next theorem, we prove that when restricted to the class of integral domains where any two elements have a maximal common divisor, the notions of being pre-Schreier and that of having the GL-property are equivalent.
\smallskip

Let $D$ be an integral domain. Following Cohn~\cite{pC68}, we say that a nonzero nonunit $p \in D$ is \emph{primal} if whenever $p \mid_D rs$ for some $r,s \in D$, we can write $p = r' s'$ for some $r', s' \in D$ such that $r' \mid_D r$ and $s' \mid_D s$.

\begin{defn}
    An integral domain is called \emph{pre-Schreier} if every nonzero nonunit of $D$ is primal.
\end{defn}

The notion of a pre-Schreier domain was coined and studied by Zafrullah in~\cite{mZ87} (although pre-Schreier domains were previously considered in various papers, including~\cite{pC68,DS78,MR78}). It turns out that the property of being pre-Schreier is stronger than the GL-property (indeed, in \cite[Proof of Lemma~2.1]{mZ90} it is proved that every pre-Schreier domain has the PSP-property, which is known to be stronger than the $D$-property). We proceed to argue that, in the class of MCD-domains, these three properties are equivalent.

\begin{theorem}
	Let $D$ be an MCD-domain. Then the following conditions are equivalent.
	\begin{enumerate}
		\item[(a)] $D$ is a pre-Schreier domain.
		\smallskip
		
		\item[(b)] $D$ is a PSP-domain.
		\smallskip
		
		\item[(c)] $D$ is a GL-domain.
	\end{enumerate}
\end{theorem}

\begin{proof}
	(a) $\Rightarrow$ (b) $\Rightarrow$ (c): This is known.
	\smallskip
	
	(c) $\Rightarrow$ (a): Assume that $D$ is a GL-domain. Fix an element $p \in D \setminus D^{\times}$ such that $p \mid_D ab$ for some $a, b \in D$. We argue that $p$ is primal. First, observe that either $\gcd(p,a) \neq 1$ or $\gcd(p,b) \neq 1$. Otherwise, the fact that $D$ is a GL-domain implies that $\gcd(p,ab)=1$, which is a contradiction.
    
    Since $D$ is an MCD-domain, we can assume, without loss of generality, that there is some nonunit $d \in D$ such that $d \in$ $\mcd_D(p, a)$. Now, pick $p' \in D$ and $a' \in D$ such that $p=d \cdot p'$ and $a = d \cdot a'$. If $p' \in D^{\times}$, it follows that $p \mid_D a$, which implies that $p$ is primal, and then we are done. 
    
    Let us instead assume that $p'$ is not a unit of $D$. The relation $p' \mid_D a'b$ follows naturally from $p \mid_D ab$. Also, the fact that $d \in \mcd_D(p, a)$ implies that $\gcd(p', a') = 1$. Notice that, as a consequence, $\gcd(p', b)\neq 1$, since otherwise $\gcd(p',a'b)=1$, contradicting that $p' \notin D^{\times}$. We can argue as before and ensure the existence of an element $d' \notin D^{\times}$ such that $d' \in \mcd_D(p', b)$. Now, pick $r\in D$ and $b' \in D$ such that $p'= d' \cdot r$ and $b= d' \cdot b'$. Notice that $\gcd(r, b')=1$. We also know that $\gcd(r, a')=1$ because $\gcd(p', a')=1$. The fact that $D$ is a GL-domain implies that $\gcd(r, a'b')=1$, and since $r \mid_D a'b'$, we can conclude that $r \in D^{\times}$. Then $p' \mid_D b$, which implies that $p$ is primal because we can write $p=d \cdot p'$ with $d \mid_D a$ and $p' \mid_D b$. 

    Hence we can now conclude that $D$ is a pre-Schreier domain.

\end{proof}

\medskip
\subsection{A Class of GL-Domains}

We proceed to identify a class consisting of GL-domains that are not pre-Schreier.

\begin{prop} \label{prop:GL not pre-Schreier domains}
	Let $D$ be an integral domain that is not a field, and let $F$ be the field of fractions of $D$. For an extension field $K$ of $F$ the following statements hold.
	\begin{enumerate}
		\item If $D$ is a UFD, then $D + x K[x]$ is a GL-domain.
		\smallskip
		
		\item If $K$ is a proper extension of $F$, then $D + x K[x]$ is not a pre-Schreier domain.
	\end{enumerate} 
\end{prop}

\begin{proof}
	Set $R := D + x K[x]$.
	\smallskip
	
	(1) To prove that $R$ is a GL-domain, fix a nonunit $r(x) \in R^*$ and let us show that $r(x)$ is prime-like. To do so, suppose that $r(x) \mid_R b(x)c(x)$ for some $b(x),c(x) \in R^*$, and let us argue that there is a nonunit divisor of $r(x)$ in $R$ that divides either $b(x)$ or $c(x)$ in  $R^*$. If $b(x) \in R^\times$ (resp., $c(x) \in R^\times$), then $r(x) \mid_R c(x)$ (resp., $r(x) \mid_R b(x)$). Therefore we can assume that neither $b(x)$ nor $c(x)$ is a unit. We divide the rest of our proof into cases.
	\smallskip
	
	\textsc{Case 1:} $\text{ord} \, r(x) \ge 1$. Observe that, in this case, each $d \in D^* \setminus D^\times$ is a nonunit divisor of $r(x)$ in $R$. In addition,  $\text{ord} \, r(x) \ge 1$ ensures that either $\text{ord} \, b(x) \ge 1$ or $\text{ord} \, c(x) \ge 1$, whence any  $d \in D^* \setminus D^\times$  is a nonunit divisor of $r(x)$ in $R$ such that either $d \mid_R b(x)$ or $d \mid_R c(x)$. Hence $r(x)$ is prime-like.
	\smallskip
	
	\textsc{Case 2:} $\text{ord} \, r(x) = 0$. We split the rest of our argument into two subcases, and conclude in each of them that $r(x)$ must be a prime-like element in $R$. This will imply that $R$ is a GL-domain, as desired.
	\smallskip
	
	\textsc{Case 2.1:} $r(0) \notin D^\times$. In this case, $r(0)$ is divisible in $D$ by an element $p \in D$ that is prime (this is because $D$ is a UFD and so $r(0)$ must factor in $D$ into finitely many primes). From the fact that $p \mid_D r(0)$, one immediately obtains that $p \mid_R r(x)$. Observe that the divisibility relation $r(x) \mid_R b(x)c(x)$ implies that $r(0) \mid_D b(0)c(0)$ and, therefore, $p \mid_D b(0)c(0)$. Since $p$ is prime in $D$, we obtain that $p \mid_D b(0)$ or $p \mid_D c(0)$. Thus, $p$ is a nonunit divisor of $r(x)$ in $R$ such that either $p \mid_R b(x)$ or $p \mid_R c(x)$. Hence we have shown that $r(x)$ is prime-like if $r(0) \notin D^\times$. 
	\smallskip
	
	\textsc{Case 2.2:} $r(0) \in D^\times$. Since $D^\times = R^\times$, after replacing $r(x)$ by $r(x)/r(0)$, we can further assume that $r(0) = 1$. As $r(x)$ is not a unit, $\deg r(x) \ge 1$ and so $r(x)$ is not a unit in $K[x]$. Let $p'(x) \in K[x]$ be a prime in $K[x]$ dividing $r(x)$, which exists because $K[x]$ is a UFD and $r(x)$ is not a unit. After replacing $p'(x)$ by $p'(x)/p'(0)$, one can further assume that $p'(0) = 1$ (because $K[x]^\times = K^\times$). As $p'(x) \mid_{K[x]} r(x)$, it follows that $p'(x) \mid_{K[x]} b(x)c(x)$. Now the fact that $p'(x)$ is prime in $K[x]$ ensures that either $p'(x) \mid_{K[x]} b(x)$ or $p'(x) \mid_{K[x]} c(x)$. Let us assume first that $p'(x) \mid_{K[x]} b(x)$. Then $p'(x)q(x) = b(x)$ for some $q(x) \in K[x]$. Because $p'(0) = 1$, we see that $q(0) = p'(0) q(0) = b(0) \in D$. Hence $q(x) \in R$, which implies that $p'(x) \mid_R b(x)$. Similarly, we obtain that $p'(x) \mid_R c(x)$ if we assume that $p'(x) \mid_{K[x]} c(x)$. Hence we have found a nonunit divisor $p'(x)$ of $r(x)$ dividing either $b(x)$ or $c(x)$, from which we conclude that $r(x)$ is also a prime-like element of $R$ when $r(0) \in D^\times$.
	\smallskip
	
	(2) It suffices to argue that $x$ is not a primal element in $R$. Since $K$ is a proper extension field of $F$, we can fix $\alpha \in K \setminus F$. It is clear that $x \mid_R (\alpha x) (\alpha^{-1}x)$. Notice that the only way $x$ can be expressed as the product of two elements of $R$ is as follows: $x = d(d^{-1}x)$ for some $d \in D^*$. However, if $d^{-1}x \mid_R \alpha x$ or $d^{-1}x \mid_R \alpha^{-1} x$, then we could take $e \in D^*$ such that $\alpha \in \big\{ \frac{e}d, de \big\} \subset F$, which contradicts that $\alpha \notin F$. Hence $x$ is not primal, which implies that $R$ is not a pre-Schreier domain.
\end{proof}

We conclude with the following concrete example, which was the initial motivation behind Proposition~\ref{prop:GL not pre-Schreier domains}.

\begin{example}
    Consider the subring $R := \zz + x\rr[x]$ of the polynomial ring $\rr[x]$. Since $\zz$ is a UFD whose field of fractions is strictly contained inside the field $\rr$, part~(1) of Proposition~\ref{prop:GL not pre-Schreier domains} implies that $R$ is a GL-domain while part~(2) of the same proposition implies that $R$ is not a pre-Schreier domain.
\end{example}

\bigskip
\section{The Ascent of the IDF-Property in the Class of GL-Domains}
\label{sec:ascent of the IDF-property over GL-domains}

The main goal of this section is to investigate whether the IDF property ascends from a GL-domain to its polynomial extension. Toward that objective, we construct a GL-domain that is antimatter and satisfies the IDF property. 

We begin our construction by exhibiting a GCD-domain with some desired conditions. Let $\{X, Y, Z, T_1, T_2, T_3, \ldots\}$ be an infinite set of algebraically independent indeterminates over $\mathbb{F}_2$. Let us consider the integral domain
\begin{equation}\label{R_0}
    R_0 := \ff_2 \bigg[ \big\{ X^\alpha, Y^\alpha, Z^\alpha : \alpha \in \qq_{\ge 0} \big\} \cup \Big \{\prod_{i=1}^{\infty} T_i^{\alpha_i} : \alpha_i \in \mathbb{Q}_{\ge 0} \ \forall \, i \in \mathbb{N} \Big\} \bigg].
\end{equation}
Set $P := \prod_{i=1}^{\infty}\qq_{\ge 0}$, and observe that the integral domain $R_0$ in \eqref{R_0} is isomorphic to the monoid algebra of the monoid $P$ over the field $\ff_2$. As $\qq_{\ge 0}$ is a valuation monoid, it is a GCD-monoid, and so $P$ is also a GCD-monoid. This, along with the fact that $\ff_2$ is a field, ensures that $R_0$ is a GCD-domain.

For the set of monomials
\[
  J := \bigg\{ \prod_{i=1}^{\infty} T_i^{\alpha_i} : \alpha_i \in \mathbb{Q}_0^+ \ \forall \, i \in \mathbb{N} \bigg\},
\]  
we consider the following subring of $R_0$:
\begin{equation}\label{main domain}
    R :=\mathbb{F}_2\big[\big\{X^\alpha, Y^\alpha, Z^\alpha : \alpha \in \qq_{\ge 0} \big\} \cup 
     \big\{X^{\alpha}T, Y^\alpha T, Z^\alpha T : (\alpha,T) \in \qq_{>0} \times J \big\} \big].
\end{equation}

We now proceed to prove that the domain $R$ is not MCD-finite.

\begin{prop} \label{prop:R is a GL-domain}
    The integral domain $R$ is a GL-domain that is not MCD-finite.
\end{prop}

\begin{proof}
    Let us first prove that the integral domain $R$ is a GL-domain. In order to do so, it is convenient to establish the following claim.
    \smallskip

    \noindent \textsc{Claim.} If $f,g \in R_0$ such that $fg \in R$, then either $f \in R$ or $g \in R$.
    \smallskip

    \noindent \textsc{Proof of Claim.} First observe that the set
    \[
        I := \langle \{X^{\alpha} : \alpha \in \mathbb{Q}^+\}, \{Y^{\alpha} :  \alpha \in \mathbb{Q}^+\}, \{Z^{\alpha} : \alpha \in \mathbb{Q}^+\} \rangle.
    \]
    is an ideal of $R$. Moreover, every element $f \in R_0$ can be written as $f = f_1 + f_2$, where $f_1 \in I$ and $f_2 \in \ff_2[ \, {T : T \in J} \, ]$. Let $f,g \in R_0$ such that $fg \in R$, and write $f=f_1+f_2$ and $g=g_1+g_2$, where $f_1, g_1 \in I$ and $f_2, g_2 \in  \ff_2[ \, {T : T \in J} \, ]$. Note that $$fg=(f_1+f_2)(g_1+g_2) = f_1g_1+f_1g_2+f_2g_1+f_2g_2.$$ Since $I$ is an ideal, we see that $f_1g_1 + f_1g_2 + f_2g_1 \in I \subset R$ and then we obtain that $f_2g_2=fg-(f_1g_1+f_1g_2+f_2g_1) \in R$. However, this leaves us with only two options: 
	\begin{itemize}
	    \item $f_2g_2=0_R$, which means that either $f_2=0_R$ or $g_2=0_R$, and this implies that either $f \in R$ or $g\in R$.
        \smallskip

        \item $f_2g_2=1_R$, which means that $f_2=g_2=1_R$ because $R$ is reduced, and then, we obtain that both $f\in R$ and $g \in R$.
	\end{itemize}
    Hence the claim is established.
    \medskip

    We are in a position to prove that $R$ is a GL-domain. Our proof is very similar to the one provided in \cite[Proposition~2.8]{AS75}. As an integral domain $D$ satisfies the GL-property if and only if the product of two primitive ideals of $D$ is also primitive (see \cite[Proposition~1.2]{AS75}), we only need to show that the product of two primitive ideals of $R$ is primitive. Let $A$ and $B$ be primitive ideals of $R$, and let us prove that $AB$ is a primitive ideal. Write 
    \[
        A := \langle a_1, a_2, \ldots , a_m \rangle \qquad \text{and} \qquad B := \langle b_1, b_2 \ldots , b_n \rangle,
    \]
    for some $a_1, \ldots, a_m \in R$ and $ b_1,\ldots,b_n \in R$. Suppose, towards a contradiction, that the ideal $AB$ is not primitive. Then there exists a nonunit $c \in R$ such that $AB \subseteq cR$. This means that there is a nonunit $c$ of $R_0$ such that $(AR_0)(BR_0) \subseteq cR \subseteq cR_0 \neq R_0$. Since $R_0$ is a GCD-domain, it is also a GL-domain, and this implies that $AR_0$ and $BR_0$ cannot both be primitive in $R_0$, so at least one of them is contained in a proper ideal of $R_0$. Set 
    \[
        \lambda := \gcd (a_1, \ldots, a_m) \qquad \text{and} \qquad \mu := \gcd(b_1, \ldots, b_n).
    \]
    Then $AR_0 \subseteq \lambda R_0$ and $ BR_0 \subseteq \mu R_0$. Moreover, at least one of $\lambda$ or $\mu$ is not a unit. Since $R_0$ is a GCD-domain, the GCD of any set of elements exists, and the following identity holds:
    \[
        \gcd(a_ib_1,\ldots, a_ib_n) = a_i \cdot \gcd(b_1,\ldots, b_n) = a_i\mu.
    \]
    Moreover, since $AB \subseteq cR_0$, we know that $c$ divides each $a_ib_j$. In particular, $c$ is a common divisor of $a_ib_1,\ldots,a_ib_n$, and thus $c \mid_R a_i\mu$ for every $i \in \ldb 1,m \rdb$. We can also note that
    \[
        \gcd(a_1\mu,\ldots,a_m\mu) = \mu \cdot \gcd(a_1,\ldots, a_m) = \mu \lambda.
    \]
    As shown above, $c$ is a common divisor of $a_1\mu, a_2\mu,\ldots, a_m\mu$ and, therefore, $c \mid_{R_0} \mu \lambda$. Since $R_0$ is a GCD-domain, it is also a pre-Schreier domain. Thus, there exist $\alpha, \beta \in R_0 $ such that $c=\alpha \beta$, with $\alpha \mid \lambda$ and $\beta \mid \mu$. Then it follows from the established claim that at least one of $\alpha$ or $ \beta$ belongs to $R$. Without loss of generality, let us assume that $\alpha \in R$. We split our analysis into the following two cases.
    \smallskip
    
    \textsc{Case 1:} $\alpha$ is a unit. Then $\beta = c$, and hence $\beta \in R$ is a nonunit. This implies that $BR_0 \subseteq \beta R_0$, a proper principal ideal of $R_0$, generated by an element of $R$.
    \smallskip
        
    \textsc{Case 2:} $\alpha$ is a nonunit. In this case, $AR_0 \subseteq \alpha R_0$, a proper principal ideal of $R_0$, generated by an element of $R$. 
    \smallskip

    We have shown that one of $AR_0$ or $BR_0$ must be contained in a proper ideal of $R_0$ generated by an element of $R$. Without loss of generality, we can assume that $AR_0 \subseteq \alpha R_0$ for some $\alpha \in R \setminus R^{\times}$.  
	
    We can then write each generator of $A$ as $a_i = \alpha j_i$ for each $i \in \ldb 1,m \rdb$, where $j_i \in R_0$. Let us write $\alpha = \alpha_1 + \alpha_2$ and $j_i = j_{i,1}+j_{i,2}$, where $\alpha_1, j_{i,1} \in I$ and $\alpha_2, j_{i,2} \in \mathbb{F}_2[ T : T\in J]$. Since $\alpha \in R$, we are done once we argue the following two possibilities for the values of $\alpha_2$.
    \begin{itemize}
        \item $\alpha_2 = 1$, and then $ j_{i,2} \cdot 1 \in R$, which implies that $j_i \in R$ and $a_i \in \alpha R$. Thus, $A \subseteq \alpha R$.
        \smallskip
    
        \item  $\alpha_2 = 0$. This implies that $\alpha \in I$. Since every element of $I$ has a square root in $I$, there exists $\delta \in I$ such that $\alpha = \delta^2$. Then $a_i= \delta^2j_i = \delta(\delta j_i)$, and therefore $\delta j_i \in \delta R_0 \subseteq IR_0 = I \subseteq R$. It follows that $a_i \in \delta R$ and $A \subseteq \delta R$.
    \end{itemize}
\end{proof}



    

We conclude this paper with our main result.
	
\begin{theorem}
    The IDF property does not ascend to polynomial extensions in the class of GL-domains.
\end{theorem}

\begin{proof}
    Let $R$ be the integral domain introduced in \eqref{main domain}, which is a GL-domain in light of Proposition~\ref{prop:R is a GL-domain}.

    Let us argue now that $R$ is not MCD-finite. Set $s_y = XY\prod_{i \in \nn} T_i$ and $s_z = XZ \prod_{i \in \nn} T_i$. Also, define a subset $B \subset R$ by
    \[
        B := \{b_i \in R : i \in \mathbb{N},\ b_i = XT_1T_2 \cdots T_{i-1}T_{i+1}T_{i+2} \cdots\}.
    \]
    As $(s_y,s_z) = b_i(YT_i, ZT_i)$ for every $i \in \nn$, the set $B$ is an infinite subset of $R$ consisting of common divisors of the set $\{s_y, s_z\}$.
    Since $R_0 \cong \ff_2[P]$ and the group of monomials in the monoid algebra $\ff_2[P]$ is a multiplicative closed set, for each $i \in \nn$, any common divisor of $YT_i$ and $ZT_i$ must be a monomial in $R_0$. However, only monomials that do not involve $X, Y, Z$ can divide $YT_i$ and $ZT_i$ simultaneously. As $R$ contains no monomials of this form, we conclude that each $b_i$ is an MCD of the set $\{s_y,s_z\}$ in $R$. Thus, $R$ is not MCD-finite.
    
    Since $R$ is an antimatter domain, it is trivially an IDF-domain. It follows from \cite[Theorem~2.1]{EK18} that the polynomial extension $R[x]$ cannot be an IDF-domain because $R$ is not MCD-finite. Thus, we have constructed a GL-domain having the IDF property such that $R[x]$ does not have the IDF property, and so we conclude that the IDF property does not ascend to polynomial extensions in the class of GL-domains.
\end{proof}

\bigskip
\section*{Acknowledgments}

This paper is the result of a collaboration carried out while the authors were part of CrowdMath 2024, a year-long free online program in mathematical research generously hosted by the MIT Mathematics Department and the Art of Problem Solving. The authors are grateful to their CrowdMath research mentor, Felix Gotti, as well as the advisors and organizers of CrowdMath for making this research opportunity possible.

\bigskip
\section*{Conflict of Interest Statement}

On behalf of all authors, the corresponding author states that there is no conflict of interest related to this paper.

\bigskip

\end{document}